\documentclass[11pt]{article}

\usepackage{amsmath,amssymb,amscd,amsthm,amsfonts}
\usepackage{graphicx,subfigure}
\usepackage{hyperref}
\usepackage{pstricks}
\usepackage{dsfont}
\usepackage{pifont}

\newtheorem{theorem}{Theorem}[section]

\newtheorem*{theoremp}{Theorem}
\newtheorem{lemma}[theorem]{Lemma}

\newtheorem{corollary}[theorem]{Corollary}
\newtheorem{conjecture}[theorem]{Conjecture}

\newtheorem{definition}[theorem]{Definition}

\newcommand{\R}{\mathds{R}}

\newcommand{\Z}{\mathbb{Z}}

\newcommand{\F}{\mathcal{F}}
\newcommand{\h}{{\mathds H}}
\newcommand{\tv}{{\mathds T}}
\newcommand{\cf}{{\cal F}}
\newcommand{\ck}{{\cal K}}

\DeclareMathOperator{\conv}{conv}

\title{Quantitative Tverberg theorems over lattices and other discrete sets}

\author{J. A. De Loera \and R. N. La Haye \and D. Rolnick \and P. Sober\'on}

\begin{document}

\maketitle

\abstract{
This paper presents a new variation of Tverberg's theorem.  Given a discrete set $S$ of $R^d$, we study the number of points of $S$ needed to guarantee the existence of an $m$-partition of the points such that the intersection of the $m$ convex hulls of the parts contains at least $k$ points of $S$.  The proofs of the main results require new quantitative versions of Helly's and Carath\'eodory's theorems.
}

\section{Introduction}

This year marks the 50th anniversary of Tverberg's theorem, one of the most important results in combinatorial convex geometry.

\begin{theoremp}[Tverberg 1966~\cite{Tverberg:1966tb}]
\label{thm:Tverberg}
 Let $a_{1},\ldots,a_{n}$ be points in $\R^{d}$.
If the number of points $n$ satisfies $n >(d+1)(m-1)$, then they can be partitioned into $m$ disjoint parts $A_{1},\ldots,A_{m}$ in such
a way that the $m$ convex hulls $\conv A_1, \ldots, \conv A_m$ have a point in common.
\end{theoremp}

The case of $m=2$ was proved in 1921 by Radon \cite{originalRadon} and is often referred to as Radon's theorem or Radon's lemma.  Tverberg published the general theorem in 1966 \cite{Tverberg:1966tb} and presented another proof in 1981 \cite{Tverberg1981}.  Simpler proofs have since appeared in \cite{baranyonn-colorfulLP, Roudneff, Sarkaria:1992vt}.  Chapter 8.3 of \cite{Mbook} and the expository article \cite{3nziegler} can give the reader a sense of the abundance of work surrounding this elegant theorem.

Traditionally, Tverberg-type theorems consider intersections of convex sets over $\R^d$.  In this article, we present Tverberg-type theorems where all points lie within a discrete subset $S \subset \R^d$ and the intersection of convex hulls is required to have a non-empty intersection with $S$.  Recall that a set $S$ is \emph{discrete} if every point $x\in \R^d$ has a neighborhood which intersects $S$ in a finite set.  Lattices, such as $\Z^d$, are important examples, but other more sophisticated discrete sets are also of interest (e.g., the difference  of a lattice and a sublattice or the Cartesian product of the primes numbers).

One motivation for considering discrete sets $S$ is that we are able to count the number of points of $S$ occurring within a finite set. This allows us to generalize the 
traditional Tverberg theorem to discrete \emph{quantitative} Tverberg theorems, in which the intersections of sets are constrained to contain at least a certain number of points. As an example, consider the case of $S=\Z^d$, which will follow as an easy corollary from our main result, Theorem \ref{thm:quantitative-disc-tverberg}.

\begin{corollary}[Discrete quantitative Tverberg over $\Z^d$] 
\label{cor:disc-quant-tv-Z}
Let $d$ be the dimension and $k$ a positive integer. Set $c(d,k)= (2^d-2)\left\lceil \tfrac23(k+1)\right\rceil+2$.  
Then, any set of at least $c(d,k) (m-1)kd+k$ integer lattice points in $\Z^d$ can be partitioned into $m$ disjoint subsets such 
that the intersection of their convex hulls contains at least $k$ integer lattice points.
\end{corollary}

For example, in the usual (real-valued) Tverberg theorem, if one has seven points with real coordinates in the plane, they can be partition into three disjoint parts such that their convex hulls intersect in at least one (real-valued) point. In our setting, if one takes 25 lattice points in $\Z^2$, there is always a $3$-partition of them, so that the three 
convex hulls intersect, and their intersection contains at least one lattice point. Prior work on this type of statement was pioneered in \cite{Eckhoff:2000jw,Jamison:1981wz, onn+radon}. To state our main results formally, we must define the (quantitative) Tverberg and Helly numbers over discrete sets $S$.

\begin{definition}
Given a discrete subset $S$ of $\R^d$, the \emph{quantitative $S$-Tverberg number} $\tv_{S}(m,k)$ (if it exists) is the smallest positive integer with the following property. For any $\tv_{S}(m,k)$ distinct points in $S\subseteq\R^d$, there is a partition of them into $m$ sets $A_1,A_2,\dots,A_m$ such that the intersection of their convex hulls contains at least $k$ points of $S$.

If no such integer exists, we say that $\tv_S(m,k)=\infty$.
\end{definition}

Recall now the classical theorem of Helly. It says that given $\cf$, a finite family of convex sets of $\R^d$, If $\bigcap \ck \neq
\emptyset$ for all $\ck \subset \cf$ of cardinality at most $d+1$, then $\bigcap \cf \neq \emptyset$. In this case $d+1$ is the Helly number of
the space $\R^d$. More generally one can define.

\begin{definition}
Given a discrete set $S \subset \R^d$, the \emph{quantitative $S$-Helly number} $\h_S(k)$ (if it exists) is the smallest positive integer with the following property. Suppose that $\mathcal F$ is a finite family of convex sets in $\R^d$, and that $\bigcap\mathcal G$ intersects $S$ in at least $k$ points for every subfamily $\mathcal G$ of $\mathcal F$ having at least $\h_S(k)$ members. Then $\bigcap\mathcal F$ intersects $S$ in at least $k$ points.

If no such integer exists, we say that $\h_S(k)=\infty$.
\end{definition}

For simplicity, we write $\tv_S(m)$ and $\h_S$ for $\tv_S(m,1)$ and $\h_S(1)$, respectively.  Note that for $k\ge 1$, the definitions of $\tv_{S}(m,k)$ and $\h_S(k)$ make sense only because $S$ is discrete.  As our main result, we prove that the existence of $S$-Tverberg numbers is a consequence of the existence of $S$-Helly numbers, and that these numbers are related.  Our main theorem below is proven in Section \ref{section-Tverberg}. As we see later,
to prove our Tverberg-type results, we will need discrete generalizations of Carath\'eodory's theorem and Helly's theorem. 

\begin{theorem}[Discrete quantitative Tverberg] \label{thm:quantitative-disc-tverberg}
Let $S\subseteq \R^d$ be discrete with finite quantitative Helly number $\h_S(k)$. Let $m,k$ be integers with $m,k\ge 1$. Then, we have
\[
\tv_S(m,k)\leq \h_S(k)(m-1)kd+k.
\]
\end{theorem}

Corollary \ref{cor:disc-quant-tv-Z} follows from Theorem \ref{thm:quantitative-disc-tverberg}, by setting $S=\Z^d$ and applying the results in result of Averkov et al.~\cite{averkov2016}, as will be discussed further in Section \ref{section-helly}.  The special case of $S=\Z^d$ and $k=1$ was considered previously in Eckhoff \cite{Eckhoff:2000jw} and in the case of $m=2$ (Radon partitions) by Onn \cite{onn+radon}.  Eckhoff described the bounds $$2^d (m-1) < \tv_{\Z^d}(m) \le (m-1)(d+1)2^d - d - 2,$$ where the upper bound follows by combining a theorem of Jamison for general convexity spaces \cite{Jamison:1981wz} with \cite{Doi1973}.  Our work improves slightly upon this bound, giving us:
$$\tv_{\Z^d}(m) \le (m-1)d 2^d +1.$$

For general $S$ and $k=1$, we have no enumeration and care only about a non-empty intersection over $S$.  In this case, the condition that $S$ be discrete is unnecessary in the definition of $S$-Helly and $S$-Tverberg numbers, and the conclusion of Theorem \ref{thm:quantitative-disc-tverberg} holds for \emph{any} subset $S\subseteq\R^d$.  We obtain the following corollary, using bounds on $S$-Helly numbers provided in \cite{AW2012,queretaro}.

\begin{corollary}[$S$-Tverberg numbers for interesting families]
The following bounds on Tverberg numbers hold:
\begin{itemize}
\item When $S=\Z^{d-a}\times\R^a$, we have $\tv_S(m)\leq (m-1)d (2^{d-a}(a+1))+1.$ 

\item Let $L',L''$ be sublattices of a lattice $L \subset\R^d$. Then, if $S=L \setminus (L' \cup L'')$, the Tverberg number satisfies $\tv_S(m)\leq 6(m-1)d 2^d+1$.

\item If $S\subseteq\R^d$ is a ${\mathbb Q}$-module, then we have $\tv_S(m)\le 2(m-1)d^2+1$.
\end{itemize}
\end{corollary}

In Section \ref{section-helly}, we discuss several new results on discrete quantitative Helly numbers, including the following:

\begin{theorem}[Discrete quantitative Helly and Tverberg numbers for differences of lattices] \label{quantitative-discrete-doignon} 
Let $L$ be a lattice in $\R^d$ and let $L_1,\dots,L_m$ be $m$ sublattices of $L$.
 Let $S=L \setminus (L_1\cup\dots\cup L_m)$ and $r=\textnormal{rank}(L)$. Then the quantitative $S$-Helly number $\h_S (k)$ exists and is bounded above by $\left(2^{m+1}k+1\right)^r$. Thus, by Theorem \ref{thm:quantitative-disc-tverberg}, we have:$$\tv_S(m,k)\leq \left(2^{m+1}k+1\right)^r(m-1)kd+k.$$
\end{theorem}

Earlier versions of our results were presented in \cite{ouroriginalpaper} which now has been divided in final form in the present paper and in a separate paper  \cite{continuousquant}, where we present similar results but regarding \emph{continuous} quantitative combinatorial convexity theorems (e.g., regarding volumes, instead of counting discrete points).  Additional quantitative results regarding the intersection structure of families of convex sets are shown by Rolnick and Sober\'on in \cite{Sob15}.  These are closely related to the contributions of this paper and include versions of the $(p,q)$ theorem of Alon and Kleitman \cite{Alon92pq} and the fractional Helly theorem of Katchalski and Liu \cite{Kat79frac}.

\section{Discrete quantitative Helly numbers}\label{section-helly}

Helly's theorem and its numerous extensions are of central importance in discrete and computational geometry (see
\cite{amenta2015helly,DGKsurvey63,Eckhoffsurvey93,Wen1997}).  Doignon was the first to calculate the $L$-Helly 
number for an arbitrary lattice $L$, which has since been much studied by researchers (see e.g., \cite{Bell:1977tm,Sca1977,Hoffman:1979ix,clarkson}).

\begin{theoremp}[Doignon~\cite{Doi1973}] 
Let $L$ be a rank $d$ lattice inside $\R^d$. Then, $\h_L$ exists and is at most $2^d$.
\end{theoremp}





In a related result, it was shown in \cite{AW2012} that $\h_{\Z^{a} \times \R^{b}}=(b+1)2^{a}$. 
Most relevant for the present work are results in \cite{queretaro} generalizing Doignon's theorem to discrete sets that are not lattices.

\begin{theoremp}[De Loera et al.~\cite{queretaro}] 
Let $L$ be a lattice in $\R^d$ and let $L_1,\dots,L_m$ be $m$ sublattices of $L$.  Let $R_m$ be the Ramsey number $R(3,3,\dots,3)$, i.e., the minimum number of vertices needed to guarantee the existence of a monochromatic triangle in any edge-coloring, using $m$ colors, of the complete graph $K_{R_m}$. Then the set $S = L \setminus (L_1\cup\dots\cup L_k)$ satisfies $\h_S\le (R_m-1)2^d$.

\end{theoremp}

 
In \cite{ipcoversion}, Aliev, De Loera, and Louveaux first extended Doignon's theorem to provide \emph{quantitative} Helly numbers, showing their first ever upper bound. This was later improved to $\h_L (k)\le (2^d-2)\left\lceil \tfrac23(k+1)\right\rceil +2$ for $L \subset \R^d$ a lattice of rank $d$ in \cite{alievatal}. Bounds for the quantitative $\Z^d$-Helly number have since been improved in asymptotic behavior in \cite{averkov16, chestnut2015sublinear}. They also proved some interesting properties, e.g., that $\h_{\Z^d}(k)$ is, somewhat surprisingly, not monotonic with respect to $k$. Exact values  for $\h_{\Z^d}(k)$ are known for $k=1,\dots,4$. Most recently, Averkov et al.~\cite{averkov2016} have expressed the quantitative $S$-Helly number $\h_{S}(k)$ in terms of polytopes with vertices in $S$ containing exactly $k$ points of $S$ in their interior.  

Before proving our main contribution in discrete quantitative Helly numbers, Theorem \ref{quantitative-discrete-doignon}, we should note the existence of similar \emph{colorful} Helly numbers. The conditions needed for this generalization have been used by several authors e.g., \cite{AW2012,baranymatousek}, and recently summarized in \cite{queretaro}.
 
First, we need the fact that the property \emph{``having at least $k$ points of $S$''} has a finite $S$-Helly number. Second, the property of having at least $k$ points of $S$ is \emph{monotone} in the sense that if  $K\subset K'$ and $K$ has at least $k$ points from $S$, then this implies that $K'$ has also at least $k$ points of $S$ within. Finally, the property of having at least $k$ points from $S$ is \emph{orderable}, because  for any finite family $\F$ of convex sets there is a direction $v$ such that:
\begin{enumerate}
\item For every $K\in\F$ with $|K \cap S|\geq k$, there is a containment-minimal $v$-semispace (i.e. a half-space of the form $\{x:v^T x\ge 0\}$) $H$ such that $|K\cap H\cap S|\geq k$.
\item There is a unique containment-minimal $K'\subset K\cap H$ with $|K' \cap S|\geq k$.
\end{enumerate}

In our case, the work presented in \cite{queretaro} shows that every monotone and orderable property with a well-defined Helly number must be \emph{colorable}:

\begin{theorem}[Colorful discrete quantitative Helly]
\label{colorful-helly}
Let $S$ be a discrete set in $\R^d$ with finite quantitative $S$-Helly number $N=\h_S(k)$.  Suppose that $\F_1,\dots\F_N$ are finite families of closed convex sets such that $|\bigcap {\cal G} \cap S|\geq k$ for every subfamily ${\cal G}$ satisfying $|{\cal G} \cap\F_i|=1$ for every $i$.  Then $|\bigcap\F_i \cap S|\geq k$ for some $i$.
\end{theorem}

Results such as Theorem \ref{colorful-helly} are called \emph{colorful} because we may think of each $\F_i$ as a different color class, in which case the relevant subfamilies ${\cal G}$ are those with one element in every color.  As a corollary to Theorem \ref{colorful-helly}, we immediately obtain the following by applying Theorem \ref{quantitative-discrete-doignon}.

\begin{corollary}[Colorful quantitative Helly for differences of lattices] \label{aquantitative-discrete-doignon} 
Let $L$ be a lattice in $\R^d$ and let $L_1,\dots,L_m$ be $m$ sublattices of $L$.
 Let $S=L \setminus (L_1\cup\dots\cup L_m)$.  Let $N=\left(2^{m+1}k+1\right)^r$, where $r=\textnormal{rank}(L)$, and let $\F_1, \ldots, \F_N$ be finite families of closed convex sets so that $|\bigcap {\cal G} \cap S|\geq k$ for every rainbow subfamily ${\cal G}$.  Then, there is an $i$ such that $|\bigcap\F_i \cap S|\geq k$.
\end{corollary}

To prove Theorem \ref{quantitative-discrete-doignon}, we will use the following definition based on \cite{alievetal,ipcoversion}. The condition that $S$ must be discrete is necessary if the following definition is to make sense for $k>1$.

\begin{definition}
Say that a subset $P$ of $S\subset\R^d$ is \emph{$k$-Hoffman} if 
$$\left|\bigcap_{p\in P}\conv(P\setminus\{p\})\cap S\right|<k.$$
The \emph{quantitative Hoffman number} $\h'_S(k)$ of a set $S\subset \R^d$ is the largest cardinality of a $k$-Hoffman set $P\subseteq S$.
\end{definition}



To compute the $S$-Helly number when $S$ is a discrete subset of $\R^d$, it suffices to consider (finite) families of convex polytopes whose vertices are in $S$, instead of families of arbitrary convex sets. In the work of Hoffman \cite{Hoffman:1979ix} and later Averkov \cite{Ave2013}, the Helly numbers of various sets $S$ were calculated using this approach. Hoffman proved $\h_S=\h'_S(1)$, where $\h_S$ is the $S$-Helly number. Here we extend their work for $k>1$ to take into account the \emph{cardinality} of the intersections with $S$.

\begin{lemma}
\label{qHoffman} Let $S \subset \R^d$ be a discrete set. 
The quantitative Hoffman number $\h'_S(k)$ bounds the quantitative Helly number $\h_S(k)$ as follows:
$\h'_S(k)-k+1 \leq \h_S(k)\leq \h'_S(k)$. 
\end{lemma}

\begin{proof}
We begin by showing that $\h_S(k)\geq \h'_S(k)-k+1$.
To do so, let $U\subset S$ be some finite set such that $\left|\bigcap_{u\in U}\conv(U\setminus\{u\}) \cap S \right|<k.$
By the definition of $\h'_S(k)$, $|U|\leq \h'_S(k)$.

Consider the family $\mathcal F=\{\conv(U\setminus\{u\})|u\in U\}$.
By the definition of $U$, $|\bigcap\mathcal F\cap S|<k$.
Note that if $\mathcal G$ is a subfamily of $\mathcal F$ with cardinality $|U|-k$, then $\mathcal G=\{\conv(U\setminus\{u\})|u\in U\setminus U'\}$ for some $U'\subseteq U$ of cardinality $k$. Consequently,
$$U'\subseteq\bigcap_{\substack{u\in U\setminus U'}}\conv(U\setminus\{u\})\cap S=\bigcap_{G \in \mathcal{G}} G\cap S.$$
Hence the quantitative Helly number 
$\h_S(k)$ must be greater than $|U|-k$.
That is, $\h_S(k)\geq  \h'_S(k)-k+1$.\\

To prove the other inequality, let $K_1,\dots,K_{\h_S(k)}$ be convex sets such that $\left|\bigcap_{j\neq i}K_j\cap S\right|\geq k$ for all $i\in[\h_S(k)]$ (where $[m]=\{1, \dots, m\}$)
yet $\left|\bigcap_{i\in[\h_S(k)]}K_i\cap S\right|<k$. Such a family $\{K_i\}$ exists by the definition of the quantitative Helly number. Then for all indices $i\in[\h_S(k)]$, there exists $U_i\subseteq\bigcap_{j\neq i}K_j\cap S$ with $|U_i|\geq k$.

Suppose $u\in U_i\cap U_j$ (for some $i\neq j$). Then $u\in\bigcap_i K_i\cap S$, 
so there can be no more than $k-1$ such points.
Hence, for each $i\in[\h_S(k)]$, there exists $u_i\in U_i$ such that 
$u_i\notin\bigcup_{j\neq i}U_j$.
In particular, the $u_i$ are distinct. 
Define now $U=\{u_i|i\in[\h_S(k)]\}$.
Consider $\bigcap_{u\in U}\conv(U\setminus\{u\})\cap S$.
Note that $U\setminus\{u_i\} = \bigcup_{j\neq i}\{u_j\}$. Because $u_j\in K_i$ for all $j\neq i$, $U\setminus\{u_i\}\subseteq K_i$.
Therefore,
\begin{align*}
\left|\bigcap_{u\in U}\conv(U\setminus\{u\})\cap S\right| & = \left|\bigcap_{i\in[\h_S(k)]}\conv(U\setminus\{u_i\})\cap S\right| \\
& \leq \left|\bigcap_{i\in[\h_S(k)]}K_i\cap S\right| \\
& < k.
\end{align*}
 By the definition of $\h'_S(k)$, it follows that $\h_S(k)\leq \h'_S(k)$.
\end{proof}

The following notion is easier to work with directly than the Hoffman number:
\begin{definition}
A set $P\subset S$ is \emph{$k$-hollow} if 
$$\big|(\conv(P)\setminus V(\conv(P)))\cap S\big|<k,$$ where $V(K)$ is the vertex set of $K$.
\end{definition}

To relate this notion to the Hoffman number, we have the following lemma.
\begin{lemma}
\label{qScarf}
Let $S\subset\R^d$ be a discrete set. 
Then $\h'_S(k)$ is equal to the cardinality of the largest $k$-hollow set with respect to $S$.
\end{lemma}

Lemma \ref{qScarf} is a partial generalization of Proposition 3 from \cite{Hoffman:1979ix}.

\begin{proof}
Let $P$ be a $k$-hollow subset of $S$, and note that
$$\bigcap_{p\in P}\conv(P\setminus\{p\})\subseteq\conv(P)\setminus V(\conv(P)).$$
Hence
$$\left|\bigcap_{p\in P}\conv(P\setminus\{p\})\cap S\right|\leq \big|(\conv(P)\setminus V(\conv(P)))\cap S\big| < k.$$
Thus any $k$-hollow set is also $k$-Hoffman; in particular, a maximum-cardinality $k$-hollow set is also $k$-Hoffman.

Let $h$ be the cardinality of the largest $k$-hollow set.
We will show that if $P$ is a subset of $S$ with cardinality greater than $h$, then $P$ is not $k$-Hoffman; this completes the proof.

Suppose $t>h$, and define $\mathcal P_t$ to be the family of $t$-element subsets of $S$.
Note that $\mathcal P_t$ can be partially ordered by inclusion of convex hull; i.e., $P\leq P'$ if $\conv(P)\subseteq\conv(P')$.

We will use induction.
Let $P\in\mathcal P_t$; suppose that all predecessors of $P$ have been shown to not be $k$-Hoffman (this vacuously includes the case when $P$ is minimal).
Because $|P|>h$, $|(\conv(P)\setminus V(P))\cap S|\geq k$.
For all $q\in(\conv(P)\setminus V(P))\cap S$, define $m(q)$ to be the number of elements $p\in P$ such that $q\notin(\conv(P\setminus\{p\}))$.
Let $(\conv(P)\setminus V(P))\cap S=\{q_1, q_2,\dots, q_k,\dots\}$, where $m(q_1)\leq m(q_2)\leq\cdots\leq m(q_k)\leq\cdots$.

We will show by induction that $m(q_1)=m(q_2)=\cdots m(q_k)=0$.
Let $1\leq j \leq k$.
Suppose that $m(q_i)=0$ for $i<j$ (note that this is vacuously true if $j=0$); assume for a contradiction that $m(q_j)>0$.
There then exists $p_j$ such that $q_j\notin\conv(P\setminus\{p_j\})$.
Note that $p_j$ must be a vertex of $P$, as otherwise $q_i$ would certainly be contained in $\conv(P\setminus\{p_j\})$.
It follows that the set $P'=\{q_j\}\cup P\setminus\{p_j\}$ is a strict predecessor of $P$.
By the outer induction hypothesis, $P'$ is not $k$-Hoffman.
That is, there exist at least $k$ points in $\bigcap_{p\in P'}\conv(P'\setminus\{p\})\cap S$.
Because $j\leq k$, at least one of those points, $r$, must have $m(r)>0$.
If $q_j\in\conv(P\setminus\{p\})$, then 
$$r\in\conv(P'\setminus\{p\})=\conv(\{q_j\}\cap P\setminus\{p,p_j\})\subseteq\conv(P\setminus\{p\}).$$
Furthermore, $r\in\conv(P'\setminus\{q_j\})=\conv(P\setminus\{p_j\})$.
Thus $0 < m(r) < m(q_j)$, which contradicts the fact that $m(q_j)$ is the smallest nonzero value of $m$.

Thus $m(q_1)=m(q_2)=\cdots=m(q_k)=0$.
Hence $\{q_1,q_2,\dots q_k\}\subseteq\bigcap_{p\in P}\conv(P\setminus\{p\})\cap S$, so $P$ is not $k$-Hoffman.

Because $S$ is discrete, every element of $\mathcal P_t$ has a finite number of predecessors under this ordering.
It follows by induction no element of $\mathcal P_t$ is $k$-Hoffman.
\end{proof}


Lemmas \ref{qHoffman} and \ref{qScarf} allow us to prove Theorem \ref{quantitative-discrete-doignon} by 
finding an upper bound on the largest $k$-hollow set.


\begin{proof}[Proof of Theorem \ref{quantitative-discrete-doignon}]
%
%
Let $P$ be a subset of $L\setminus\bigcup_i L_i$ with cardinality $n^r+1$, where $n=k\cdot2^{m+1}+1$.
We will show that $P$ is not $k$-hollow; this implies Theorem \ref{quantitative-discrete-doignon} by Lemma \ref{qHoffman} and the contrapositive of Lemma \ref{qScarf}.

It is a simple fact, first observed in \cite{rabinowitz}, that there must 
exist $n+1$ collinear points $z_0,z_1,\dots,z_n$ in $\conv(P) \cap L$ with $z_0,z_n\in L\setminus\bigcup_i L_i$.
Note that $z_{j}$ and $z_{j+\ell}$ cannot both be in $L_i$ if $j=0\mod\ell$; otherwise, $z_0$ would be in $L_i$.
Suppose $z_i$ is in some sublattice $L_j$ for all $i\in[\ell\cdot 2^m,(\ell+1)2^m]$.
By the above note, 
$z_{\ell\cdot 2^m+2^a}$ and $z_{\ell\cdot 2^m+2^b}$ cannot both be in the same sublattice for any $0\leq a<b\leq m$.
This is impossible, as there are $m+1$ values of $a$ and only $m$ sublattices $L_i$.
Therefore $z_i\in L\setminus\bigcup_jL_j$ for some $i\in[\ell\cdot 2^m, \ell\cdot 2^m+1,\dots,(\ell+1)\cdot2^m]$.
Since $n=k\cdot 2^{m+1}+1$, 
$$\left\{\left\{z_i\middle|i\in\left[(2\ell-1) 2^m,2\ell\cdot2^m\right]\right\}\middle|1\leq\ell\leq k\right\}$$
is a family of $k$ disjoint subsets of $\{z_1,\dots,z_{n-1}\}$, each of which contains a point in $L\setminus\bigcup_jL_j$.
It follows that $\{z_1,\dots,z_{n-1}\}$ contains at least $k$ elements of $L\setminus\bigcup_jL_j$.
None of these points can be vertices of $\conv(P)$, as they are strictly between the two endpoints.
Hence
$$\left|\left(\conv(P)\setminus V(\conv(P))\right)\cap L\setminus\bigcup_jL_j\right|\geq k.$$
That is, $P$ is not $k$-hollow.
Consequently, no $k$-hollow set $P$ can have size greater than $\left(2^{m+1}k+1\right)^r$.
Therefore $\h_S(k)\leq \h'_S(k)\leq\left(2^{m+1}k+1\right)^r$.
\end{proof}

We end by remarking that using Theorem 3 of the recent paper \cite{averkov2016} one can derive a similar result as our Theorem \ref{quantitative-discrete-doignon}.

{\section{Proof of Theorem \ref{thm:quantitative-disc-tverberg}} \label{section-Tverberg}

Recall now the classical 1907 theorem of Carath\'eodory \cite{originalCaratheodory}.
Let $S$ be any subset of $\R^d$. Then each point in the convex hull of $S$ is a convex combination of at most $d+1$ points of $S$.
Our proof of Theorem \ref{thm:quantitative-disc-tverberg} requires generalizations of Carath\'eodory's theorem.

\begin{theoremp}[Very colorful Carath\'eodory theorem, B\'ar\'any 1982 \cite{baranys-caratheodory}] \label{verycolorfulcaratheodory}
Let $X_1, X_2, \ldots, X_d \subset \R^d$ be sets, each of whose convex hulls contains $p\in \R^d$ and let $q \in \R^d$.  Then, we can choose $x_1 \in X_1, \ldots, x_d \in X_d$ such that
	\[
	p \in \conv\{x_1, x_2, \ldots, x_d, q\}.
	\] 
\end{theoremp}

We will use this result to prove the following extension.

\begin{lemma}[Colorful discrete quantitative Carath\'eodory]\label{theorem-quantitative-discrete-caratheodory}
	Let $K$ be the convex hull  of $m \ge 2$ points in $\R^d$, and $\operatorname{ex}(K)$ be the number of extreme points of $K$.  
	If $n=\operatorname{ex}(K)$ and $X_1, X_2, \ldots, X_{nd}$ are sets whose convex hulls contain $K$, then we can find $x_1 \in X_1, \ldots, x_{nd} \in X_{nd}$ such that
	\[
	K \subset \conv\{x_1, \ldots, x_{nd}\}.
	\]
	Moreover, the number of sets is optimal for the conclusion to hold.
\end{lemma}

We believe this result may already be known, but we have not found references to it.  
Lemma \ref{theorem-quantitative-discrete-caratheodory} is ``colorful'' in the following sense: given sets $X_1, \ldots, X_{nd}$, considered as color classes, whose convex hulls contain a set $K$ with $n$ vertices, we want to make a colorful choice $x_1 \in X_1, \ldots, x_{nd} \in X_{nd}$ such that $\conv \{x_1, \ldots, x_{nd}\}$ also contains $K$. 

\begin{proof}[Proof of Lemma \ref{theorem-quantitative-discrete-caratheodory}]
Let $K$ be a polytope with $n$ vertices $y_1,y_2,\dots,y_n$; let $X_1,X_2,\dots, X_{nd}$ be sets whose convex hulls contain $K$.
Without loss of generality, we may assume th at $K$ contains the origin in its relative interior.

Using the very colorful Carath\'eodory theorem above, 
for a each $j$ we can find $x_{(j-1)d+1}\in X_{(j-1)d+1}, \dots,x_{jd}\in X_{jd}$ such that
$$y_j\in\conv\{0, x_{(j-1)d+1},\dots,x_{jd}\}.$$
To finish the proof, it suffices to show that $0\in\conv\{x_1,\dots,x_{nd}\}$.
If this is not the case, then there must be a hyperplane separating 0 from $\conv\{x_1,\dots,x_{nd}\}$.
We may assume that the hyperplane contains 0 and leaves $\{x_1,\dots,x_{nd}\}$ in the same open halfspace.
Because $K$ contains 0 in its relative interior, there must be a vertex $y_j$ of $K$ in the other (closed) halfspace, contradicting the fact that $y_j\in\conv\{0,x_1,\dots,x_{nd}\}$.

In order to show the value $nd$ is optimal, consider a convex polytope $K'$ which has each $y_i$ in the relative interior of one of its facets, and such that the facets corresponding to $y_i$ and $y_j$ do not share vertices for all $i \neq j$.  Then, take $nd-1$ copies $X_1, X_2, \ldots, X_{nd-1}$ of $K'$.  Any colorful choice whose convex hull contains $K$ needs at least $d$ vertices for each extreme point of $K$, which is not possible.
\end{proof}




We now turn our attention to the discrete quantitative Tverberg theorem, Theorem \ref{thm:quantitative-disc-tverberg}. We will use the notion of the depth of a point inside a set to present a cleaner argument.  We say that a point $p$ has \emph{depth} at least $a$ with respect to a set $A$ if for every closed halfspace $H^+$ containing $p$, we have $|H^+\cap A|\ge a$.  We say that a set of points $P$ has \emph{depth} at least $a$ with respect to $A$ if every $p\in P$ has depth at least $a$ with respect to $A$. We will use the following lemma in the proof of Theorem \ref{thm:quantitative-disc-tverberg}.

\begin{lemma}
\label{lem:depth}
If a set of points $P$ has depth at least 1 with respect to $A$, then the convex hull of $A$ contains $P$.
\end{lemma}
\begin{proof}
If this were not true, then there would exist some hyperplane $H$ separating $\conv(A)$ from a point $p\in P$, contradicting the definition of having depth at least one.
\end{proof}





\begin{proof}[Proof of Theorem  \ref{thm:quantitative-disc-tverberg}]
Suppose that $A\subseteq S$ contains $\h_S(k)(m-1)kd+k$ points.  We will construct an $m$-Tverberg partition of $A$. For this, consider the family of convex sets
$$\mathcal F=\left\{ \conv(B) |B\subset A,|B|=(\h_S(k)-1)(m-1)kd+k\right\}.$$

Note that for any $F\in\mathcal F$, we have $|A\setminus F|\le (m-1)kd$.  Therefore, if $\mathcal G$ is a subfamily of $\mathcal F$ with cardinality $\h_S(k)$, we must have
$$\left| A\setminus\bigcap_{G\in \mathcal G} G\right| \le \h_S(k)(m-1)kd.$$
Since there are $\h_S(k)(m-1)kd+k$ points in $A$, $\bigcap_{G \in \mathcal G} G$ must contain at least $k$ elements of $S$.  Hence, by the definition of the quantitative Helly number $\h_S(k)$, $\bigcap_{F\in \mathcal F} F$ contains at least $k$ elements of $S$. Let $P=\{p_1,p_2,\dots,p_k\}$ be $k$ of those points.

\textbf{Claim 1.} The set $P$ has depth at least $(m-1)kd+1$ with respect to $A$.

Suppose that this is not true. Then, some closed halfspace $H^+$ contains an element of $P$ and at most $(m-1)kd$ elements of $A$.  This means that there are at least $(\h_S(k)-1)(m-1)kd+k$ elements of $A$ in the complement of $H^+$. However, this means that some $F\in \mathcal{F}$ lies in the complement of $H^+$, a contradiction, since every such $F$ must contain all points of $P$.

The theorem now follows immediately from the following claim:

\textbf{Claim 2.} For each $j\le m$, we can find $j$ disjoint subsets $A_1,A_2,\dots,A_j \subset A$ such that $P\subset \conv(A_i)$ for every $i$.

We proceed by induction on $j$.  In the base case of $j=1$, Claim 1 tells us that $P$ has depth at least one with respect to $A$.  Hence, by Lemma \ref{lem:depth}, we have $P\subset \conv(A)$.  Now suppose that $j>1$.  By our inductive hypothesis, we can find $j-1$ disjoint subsets $A_1,A_2,\dots,A_{j-1}\subset A$ such that $P\subset \conv(A_i)$ for every $i$.

\textbf{Case 1.} $k\ge 2$.

Applying Lemma \ref{theorem-quantitative-discrete-caratheodory}, we may assume that $A_1,A_2,\ldots,A_{j-1}$ have cardinality at most $kd$, so the depth of $P$ is diminished by at most $kd$ if we remove $A_i$ from $A$.  It follows that the depth of $P$ is at least 1 with respect to $A\setminus \cup_{1\le i\le j-1} A_i$.  Hence, by Lemma \ref{lem:depth}, we can find $A_j\subset A$ disjoint from $A_1,A_2,\ldots,A_{j-1}$ such that $P\subset \conv(A_j)$.

\textbf{Case 2.} $k=1$.

In this case, $P=\{p\}$. By standard Carath\'eodory's theorem, we may assume that each $A_i$ (for $1\le i\le j-1$) either has cardinality less than $d+1$ or else has cardinality $d+1$ and defines a full-dimensional simplex. Notice that every halfspace containing $p$ can contain at most $d$ points of $A_i$, since $p\in \conv(A_i)$. Proceeding as in Case 1, we conclude that $p$ has depth at least 1 with respect to $A\setminus \cup_{1\le i\le j-1} A_i$.  Hence, by Lemma \ref{lem:depth}, we can find $A_j\subset A$ disjoint from $A_1,A_2,\ldots,A_{j-1}$ such that $P\subset \conv(A_j)$.

This completes our induction and proves the theorem.
\end{proof}}

\section{Future directions}

We have here defined and demonstrated the existence of quantitative Tverberg, Helly, and Carath\'eodory theorems over discrete sets.  In a related work, we will also be presenting new results on \emph{continuous} quantitative versions of these theorems.  In the continuous setting, the intersection of sets is measured not by its cardinality over a discrete set $S$, but by some continuous parameter such as the diameter or volume.

As with Helly's and Carath\'eodory's theorems (see Theorem \ref{colorful-helly} and Lemma \ref{theorem-quantitative-discrete-caratheodory}), there is the potential for colorful versions of Tverberg's theorem.  In this case, the aim is to impose additional combinatorial conditions on the resulting partition of points, while guaranteeing the existence of a partition where the convex hulls of the parts intersect. Now that the conjectured topological versions of Tverberg's theorem have been proven false \cite{frick2015counterexamples}, the following conjecture by B\'ar\'any and Larman is arguably the most important open problem surrounding Tverberg's theorem.

\begin{conjecture}[B\'ar\'any and Larman, 1992 \cite{Barany:1992tx}]\label{conjecture-colorful-Tverberg}
Let $F_1, F_2, \ldots, F_{d+1} \subset \R^d$ be sets of $m$ points each, considered as color classes.  Then, there is a colorful partition of them into sets $A_1, \ldots, A_{m}$ whose convex hulls intersect.
\end{conjecture}

Along these lines, we conjecture the following colorful discrete quantitative Tverberg theorem:

\begin{conjecture}
Let $S \subset \R^d$ be a set such that the Helly number $\h_S(k)$ is finite for all $k$.  Then, for any $m,k$ there are integers $m_1$ and $m_2$ such that the following statement holds.

Given $m_1$ families $F_1, F_2, \ldots, F_{m_1}$ families of $m_2$ points of $S$ each, considered as color classes, there are $m$ pairwise disjoint colorful sets $A_1, A_2, \ldots, A_m$ such that
\[
\bigcap_{i=1}^m \conv (A_i)
\]
contains at least $k$ points of $S$.
\end{conjecture}

\section*{Acknowledgments} We are grateful to I.~B\'ar\'any, A.~Barvinok, F.~Frick, A.~Holmsen, J.~Pach, and G.M.~Ziegler for their comments and suggestions. This work was partially supported by the Institute for Mathematics and its Applications (IMA) in Minneapolis, MN funded by the National Science Foundation (NSF). The authors are grateful for the wonderful working environment that led to this paper. The research of De Loera and La Haye was also supported first by a UC MEXUS grant and later by an NSA grant. Rolnick was additionally supported by NSF grants DMS-1321794 and 1122374.

\bibliographystyle{plain}

\bibliography{references}

\vskip .7in

\noindent J.A. De Loera and R.N. La Haye\\
\textsc{
Department of Mathematics \\
University of California, Davis \\
Davis, CA 95616
 \\
}\\[0.3cm]
\noindent D. Rolnick \\
\textsc{
 Department of Mathematics \\
 Massachusetts Institute of Technology \\
 Cambridge, MA 02139
 \\
}\\[0.3cm]
\noindent P. Sober\'on \\
\textsc{
Department of Mathematics \\
Northeastern University \\
Boston, MA 02115
}\\[0.3cm]

\noindent \textit{E-mail addresses: }\texttt{deloera@math.ucdavis.edu, rlahaye@math.ucdavis.edu, drolnick@math.mit.edu, p.soberonbravo@neu.edu}
\end{document}